\documentclass[12pt,a4paper]{article}

\usepackage{graphicx,color}
\usepackage{amsbsy,amsfonts,amsmath,amssymb,amsthm}
\usepackage{latexsym}
\usepackage{fancybox}
\usepackage{epsfig}
\usepackage{float}
\usepackage{graphicx}
\usepackage{appendix}

\usepackage{subcaption}
\usepackage{comment}
\usepackage[shortlabels]{enumitem}

\def\z{{\bf z}}

\newcommand{\R}{{\mathbb R}}

\def\1{\raisebox{2pt}{\rm{$\chi$}}}

 \newtheorem{thm}{Theorem}[section]
 
 \newtheorem{lem}[thm]{Lemma}
 \newtheorem{prop}[thm]{Proposition}
 \newtheorem{defn}[thm]{Definition}
 \newtheorem{rem}[thm]{Remark}
 \newtheorem{ex}{Example}
 \numberwithin{equation}{section}

\begin{document}

\begin{center}
{\LARGE\textbf{
   An augmented Lagrangian model for signal segmentation 
  }}
\end{center}
\vspace{7ex}
\begin{center}
{\sc Salvador Moll}\footnotemark[2]\\
Department d'An\`{a}lisi Matem\`{a}tica, Universitat de Val\`{e}ncia\\
C/Dr. Moliner, 50, Burjassot, Spain\\
{\ttfamily j.salvador.moll@uv.es}
\vspace{2ex}

{\sc Vicent Pallard\'o}\\
Universitat de Val\`encia\\ C. Dr. Moliner, 50, Burjassot, Spain \\ {\ttfamily vicentpallardojulia@gmail.com}
\end{center}
\vspace{10ex}

\noindent
{\bf Abstract.}
In this paper, we provide a new insight to the two--phase signal segmentation problem. We propose an augmented Lagrangian variational model based on Chan--Vese's original one. By using both energy methods and PDE methods, we show, in the one dimensional case, that the set of minimizers to the proposed functional contains only binary functions and it coincides with the set of minimizers to Chan--Vese's one. This fact allows us to obtain two important features of the minimizers  as a byproduct of our analysis. First of all, for  a piecewise constant initial signal, the  jump set of any minimizer is a subset of the jump set of the given signal. Secondly, all of the jump points of the minimizer belong to the same level set of the signal, in a multivalued sense. This last property permits to design a  trivial algorithm for computing the minimizers.

\footnotetext[0]{

$\empty^*$\,AMS Subject Classification
	35G60, 35Q68, 35J92, 49J10  

$\empty^\dag$\, {This author is partially supported by the Spanish MCIU and FEDER project PGC2018-094775-B-I00 and by Conselleria d'Innovaci\'o, Universitats, Ci\`encia i Societat Digital, project AICO/2021/223.}}

\section{Introduction}\label{}

Segmentation is the task of partitioning an object into its constituent parts. In signal or image processing, it consists in decomposing the domain $\Omega$ of a given input: a signal (an interval in which case $\Omega=[a,b]\subset \R$) or an image ($\Omega\subset\R^2$) into some regions of interest. In the particular case of two--phase segmentation, the aim is to find an optimal partition into two disjoint subsets, the foreground domain $\Omega_F$ and the background domain $\Omega_B$ such that $\Omega=\Omega_F \cup \Omega_B$.\\

After the seminal work by Mumford and Shah \cite{Mumford-Shah}, in which the authors introduced a celebrated variational model for image segmentation, Chan and Vese rewrote it in the two--phase framework \cite{Chan-Vese}. They propose to obtain the optimal partition by minimizing the following energy functional
\begin{equation}\label{chan-vese}
	\min_{E,c_1,c_2} {\rm Per} (E;\Omega)+\lambda_1 \int_E (c_1-f)^2\, dx+\lambda_2\int_{\Omega\setminus E} (c_2-f)^2\,dx
\end{equation}
among all sets of finite perimeter $E\subset \Omega$ and all constants $c_1,c_2\in [0,1]$, for some given parameters $\lambda_1,\lambda_2\geq 0$. From now on, we don't distinguish between the weights of the foreground and background and thus we take $\lambda_1=\lambda_2=\lambda$. A minimizer $E\subset\Omega$ can be considered as the foreground domain and $\Omega\setminus E$ as the background domain. In this case, the constants $c_1$ and $c_2$ turn out to be, the average of $f$ in $E$ and the average in $\Omega\setminus E$, respectively. The authors proposed in \cite{Chan-Vese} an iterative two step algorithm for finding the minimizers of the energy based on the level set formulation developed by Osher and Sethian \cite{Osher-Sethian}. Basically, after initialization of the constants, the first step consists in finding the corresponding minimizing set with fixed constants as a steady state solution of the correspondent $L^2$- gradient flow of the functional associated to the level set formulation. Then, one recomputes the constants and cames back to the first step until convergence has been reached.\\

The main problem with this algorithm is that the energy functional is not convex. Therefore, the gradient descent scheme is prone to get stuck at critical points other than global minima. This issue was fixed by Chan- Esedoglou and Nikolova, who proved that minimizers to Chan-Vese's level set functional with fixed constants (i.e. solutions to Step 2) are solutions to the following constraint convex energy minimization problem (and viceversa):
\begin{equation}
	\label{chan-esedoglu-nikolova}
	\min_{\{ 0\leq u\leq 1\}} |Du|(\Omega)+\lambda\int_\Omega u(c_1-f)^2\, dx+(1-u)(c_2-f)^2\, dx.
\end{equation}
Observe that, if the solution $u$ is the characteristic function of a set with finite perimeter; $u=\chi_E$, then, the energy in \eqref{chan-esedoglu-nikolova} coincides with that in \eqref{chan-vese}. There are still two main problems: the main one remains at the nonconvex nature of the original energy functional; i.e. convergence of the algorithm to a global minima is not ensured, and it heavily depends on the initialization. Moreover, it is not known if Chan-Vese's algorithm (with Chan-Esedoglu-Nikolova modification) could lead to non-binary solutions (see \cite{Chan-Esedoglu-Nikolova}).\\

The main objective of this work is to give another approach to original Chan-Vese's minimizers in the easiest possible case, the one dimensional case, which corresponds to signal segmentation. In the context of signal sementation, Chan-Vese's algortihm was already proposed in \cite{Chan-Vese-Multiphase} and has been used in some works (see \cite{Mahmoodi} or \cite{Mahmoodi2}). Our starting point is the functional appearing in Problem \eqref{chan-esedoglu-nikolova}. We aim at minimizing, simultaneously the function $u$ and the constants $c_1,c_2$. In order to do that, we introduce an augmented lagrangian version of the functional, coupled with the constraint $0\leq u\leq 1$. In this new functional, we replace the constants by BV functions while highly penalizing their variation. For any $\varepsilon>0$ we define the functional $F_\varepsilon: (L^2(0,1))^3\to [0,+\infty[$ by letting
\begin{align}\label{EAL}
F_\varepsilon(u,v_1,v_2) =& |Du|(\Omega) + \frac{1}{\varepsilon}\left(|Dv_1|(\Omega) + |Dv_2|(\Omega)\right) + \nonumber\\
			 +&	\lambda\int_\Omega u\left(v_1 - f\right)^2\,dx + (1-u)\left( v_2 - f\right)^2\,dx  + \int_\Omega \mathbb{I}_{[0,1]}(u)\,dx,
\end{align}
where $\mathbb{I}_{[0,1]}(\cdot)$ denotes the indicator function of the interval $[0,1]$; i.e. $$\mathbb{I}_{[0,1]}(x)=\left\{\begin{array}
  {cc} 0 & {\rm if \ } x\in [0,1] \\ +\infty & {\rm otherwise \ }.
\end{array}\right.$$

The second term is implemented for penalizing the variation of the { pair of functions} $(v_1, v_2)$. { Observe that, letting $\varepsilon\to 0$ we are forcing $v_1,v_2$ to become constants}. With this addition, {the} functional $F_\varepsilon$ fails to be convex. However, we can use standard PDE methods to obtain some  features of the set of minimizers via its correspondent system of Euler-Lagrange equations. In particular, we prove the following results in the case that $\Omega$ is an interval of $\R$:
\begin{thm}
  \label{th:constant}
  Let $\varepsilon<{\frac{1}{4\lambda}}$. Then, if  $(u,v_1,v_2)$ is a minimizer of $F_\varepsilon$, then $v_1,v_2$ are constants{.}
\end{thm}

{ In order to characterize the first component of the minimizer we need to assume further that the datum is not too oscillatory in the following sense:

\medskip
\noindent{$(H)$}
	${f\in BV(0,1)}$ {satisfies that} for every $c\in(0,1)$,
	$$
	\mathcal{L}^1(\partial\{x\in (0,1)\setminus J_f : f(x) = c  \}) = 0
	$$

{Observe that with this assumption we exclude some pathologies on the data such as having a fat Cantor set as a level set.}}

\begin{thm}
  \label{th:binary} Given $\varepsilon<{\frac{1}{4\lambda}}$, for $f\in BV(\Omega)$ { satisfying $(H)$}, any minimizer $(u,v_1,v_2)$ of $F_\varepsilon$ is independent of $\varepsilon$ and it satisfies that either $u$ is constant or $u(\Omega)\subset\{0,1\}$; i.e. $u$ is a binary function in $BV(\Omega)$.
\end{thm}

With this last result, we can show that the set of minimizers of Chan-Vese's problem \eqref{chan-vese} coincide with the set of minimizers of $F_\varepsilon$ (independent of $\varepsilon$ under the size condition above expressed). As a by-product of our analysis, we obtain two important properties of solutions to \eqref{chan-vese}:

\begin{itemize}
  \item[$(a)$] The jump set of any solution is concentrated in the topological boundary of a sole level set { (in a multivalued sense, see \eqref{jumplevelset} for the proper statement)}.
  \item[$(b)$] If $f$ is piecewise constant, then the jump set of any solution is contained in the jump set of $f$.
  \end{itemize}

These two properties, though quite intuitive in the one dimensional setting, were not known in the literature, to the best of our knowledge. Moreover, property $(a)$ allows to build a trivial algorithm to find the minimizers of Chan-Vese's problem in the one-dimensional case.

\medskip
The plan of the paper is the following one: In Section \ref{sec:system} we obtain the system of PDE's that minimizers to $F_\varepsilon$ satisfy; i.e. the corresponding Euler--Lagrange equations in this non-smooth case. In Section \ref{sec:thm} we prove Theorems \ref{th:constant} and \ref{th:binary}. Section \ref{sec:prop} is devoted to the proof of properties $(a)$ and $(b)$. In Section \ref{sec:algorithm}, we explain the trivial algorithm to compute the minimizers{.}
We finish the paper with some conclusions and with an Appendix, in which we collect the existence of minimizers to $F_\varepsilon$ { as well as the proof of an auxiliary result we need}.\\

{\bf \noindent Notations.} Throughout the paper, $\Omega$ denotes an open bounded set in $\R^N$ {with Lipschitz boundary} and $\mathcal L^N$ denotes the Lebesgue measure in $\R^N$. We denote by $L^p(\Omega)$, $1\leq p\leq \infty$ the Lebesgue space of functions which are integrable with power $p$ with respect to $\mathcal L^N$.  We use the notation $\langle \cdot,\cdot\rangle$ to denote the scalar product between two $L^2$ functions. We denote by $H^1(\Omega)$ the Hilbert space $W^{1,2}(\Omega)$ and by $H_0^1(\Omega)$ the completion in $H^1(\Omega)$ of smooth functions with compact support in $\Omega$.  We use standard notation for functions of bounded variation ($BV$ functions) as in \cite{Ambrosio}. In particular, given $u\in BV(\Omega)$, we write $u_x$, $D^c u$ and $D^j u$ for the absolutely continuous part of the measure $Du$ with respect to $\mathcal L^N$, for the Cantor part of $Du$ and for the jump part of $Du$, respectively. We use the notation $u^\pm(x)$ for the left and right approximate limits of $u$ at $x\in \Omega$ (we use only this convention for the case of {$\Omega=(a,b)\subset \R$, $a\leq b\in\R$}),  $J_u$ for its jump set and $\nu^u$ will denote the Radon-Nikodym derivative of $Du$ with respect to $|Du|$. Given a set $E\subseteq \Omega$ we say that it is a set of finite perimeter in $\Omega$ if $\chi_E\in BV(\Omega)$, where $\chi_E$ denotes the characteristic function of the set $E$. In this case, its perimeter is defined as $Per(E;\Omega):=|D\chi_E|$.  Finally, unless otherwise specified, we always identify a function (in {$H^1_0(0,1)$} or in $BV(\Omega)$) by its precise representative.
%
%
%
%



%
%
\section{System of Euler--Lagrange equations}\label{sec:system}

In this Section, we derive the system of equations that minimizers of $F_\varepsilon$ must satisfy. Although $F_\varepsilon$ is not a convex functional in $(L^2(\Omega))^3$, it is a convex functional in each of their coordinates when one fixes the other two ones. Therefore, by standard results in convex analysis, we obtain that the Euler--Lagrange system of PDE's is the following one:

\begin{equation}\label{descsysbefore}
\begin{cases}
\begin{split}
\quad \lambda \left( (v_1 - f)^2 - (v_2 -f)^2 \right) + \partial (\Phi+\Psi)(u)&\ni 0 \\
\quad 2\varepsilon\lambda u(v_1 -f) +\partial \Phi(v_1) &\ni 0 \\
\quad 2\varepsilon\lambda (1-u)(v_2 -f) +\partial\Phi(v_2) &\ni 0,
\end{split}
\end{cases}
\end{equation}

{\noindent where} the symbol $\partial$ denotes the subdifferential (in $L^2(\Omega)$) of the following two extended real valued convex functions: $$\Phi(g) :=\left\{\begin{array}
  {cc} |Dg|(\Omega) & {\rm if \ } g\in L^2(\Omega)\cap BV(\Omega) \\ +\infty & {\rm if \ } g\in L^2(\Omega)\setminus BV(\Omega)
\end{array}\right.,$$and
$$\Psi(g) = \int_\Omega \mathbb{I}_{[0,1]}(g)\,dx.$$

We easily note that a.e. $\partial \Psi(g) = \partial \mathbb{I}_{[0,1]}(g)$, for any $g\in L^2(\Omega)$. Next result is crucial and permits to reformulate $\partial (\Phi + \Psi)$.
\begin{thm}\label{thm:sumsubdif}
	Let $g\in L^2(\Omega)$. Then,
	$$\partial \left(\Phi + \Psi\right)(g) = \partial \Phi(g) + \partial \Psi(g)$$
	\end{thm}
\begin{proof}
  We point out that standard results to decompose the subdifferential of the sum, such as those in \cite{Brezis}, cannot be applied since {the interior of the domains of both functionals are empty.} We follow the strategy of \cite[Thm. 3.1]{Shirakawa}, consisting in an ad-hoc proof by approximating the subdifferential of the indicator function by its Yoshida's regularization. First of all, we state the following claim, whose proof we postpone to the Appendix.\\

 {\noindent\textit{{\bf Claim:} Let $h:\R\to \R$ be a Lipschitz non decreasing function. Then, {it} holds:}
  \begin{equation*}
    \langle v, h(u)\rangle = |Dh(u)|(\Omega),\quad
  \forall v\in \partial\Phi(u).\end{equation*}}

\noindent Since we know that the inclusion
\begin{equation*}\label{inclusionsubdif}
\partial\left(\Phi + \Psi\right)(u) \supseteq \partial\Phi(u) + \partial\Psi(u)\,,  \qquad \forall u\in L^2(\Omega)
\end{equation*}
is satisfied, it is sufficient to show the converse inclusion. Let $u\in L^2(\Omega)$ be such that $\partial (\Phi +\Psi)(u)\neq \emptyset$,  and let $v\in\partial (\Phi +\Psi)(u)$. We define $\varphi_u:L^2(\Omega)\rightarrow\mathbb{R}\cup\{+\infty\}$ by
$$
\varphi_u(w) = (\Phi + \Psi)(w) + \frac{1}{2}\|{w}\|_2^2 -\langle v + u, w\rangle\,.
$$
We note that $\varphi_u$ is coercive, strictly convex and lower semi-continuous. Then, it is easy to see that $u$ is the unique minimizer of $\varphi_u$.\\\\
Now, for any $0 < \eta<1$, let $\beta_ \eta$ be the Yoshida regularization of $\partial \mathbb{I}_{[0,1]}$, {i.e.}
$$
\beta_\eta(t) = \frac{(t-1)_+ - t_-}{\eta}\,, \qquad \forall t\in\mathbb{R},
$$
{here subindexes $\pm$ represent respectively, the positive and the negative part of the function.}
We next consider $\gamma: L^2(\Omega)\rightarrow \mathbb{R}\cup+\infty$ defined by
$$
	\gamma(w) = { \Phi(w) +\frac{1}{2}\|w\|_2^2-\langle v+u,w\rangle} + \int_\Omega \overline{\beta_\eta}(w)\,dx\,,
$$
where $\overline{\beta_\eta}$ is a primitive of $\beta_\eta$. We note that $\gamma$ is coercive,
strictly convex over $\text{Dom}(\Phi + \Psi)$ and lower semi-continuous. Thus, $\gamma$ has a unique minimizer $u_\eta$ which satisfies the corresponding Euler--Lagrange equation:
$$
	-\beta_\eta(u_\eta) - u_\eta + (v + u)\in \partial \Phi(u_\eta)\,.
$$
In consequence, we know that the following equation has a unique solution
$$
v_\eta + \beta_\eta(u_\eta) + u_\eta = v + u\,, \qquad \text{where $v_\eta \in \partial \Phi(u_\eta)$}\,.
$$
{Multiplying by $\beta_\eta(u_\eta)$ both sides of the previous equation}, we have, for any $0<\eta<1$
$$
\langle v_\eta, \beta_\eta(u_\eta) \rangle + \|\beta_\eta(u_\eta)\|_2^2 + \langle u_\eta, \beta_\eta(u_\eta) \rangle = \langle v + u, \beta_\eta(u_\eta)\rangle
$$
Here, since $t\beta_\eta(t)\geq 0$ for any $t\in\mathbb{R}$, we note that $\langle u_\eta, \beta_\eta(u_\eta)\rangle \geq 0$. Moreover, by the previous Claim, we have {that}
$$
\langle v_\eta , \beta_\eta(u_\eta)\rangle \geq 0\,.
$$
Then, we see that $\{\beta_\eta (u_\eta) \}_\eta$ and $\{ v_\eta\}_\eta$ are bounded in $L^2(\Omega)$ and $\{u_\eta \}_\eta$ is bounded in $L^2(\Omega)$ and $BV(\Omega)$. Therefore, there {is} a sequence $\{\eta_n \}_{n}\subset (0,1)$ such that $\eta_n\to 0$ and {there exists a} function $u_0\in L^2(\Omega)$ such that
$$
u_{\eta_n}\stackrel{n\rightarrow \infty}{\longrightarrow} u_0 \quad \text{in $L^1(\Omega)$}\,.
$$
Moreover, we can assume that there {exist} $v_0, \xi \in L^2(\Omega)$ such that $u_{\eta_n}, v_{\eta_n}$ and $\beta_{\eta_n}(u_{\eta_n})$ converge weakly in $L^2(\Omega)$ to $u_0,v_0$ and $\xi$, respectively.\\\\
In addition, we note that
$$
\frac{1}{\eta}(u_\eta -1)_+ \leq |\beta_\eta(u_\eta)|\quad\text{and}\quad \frac{1}{\eta}(u_\eta)_-\leq |\beta_\eta (u_\eta)|\,,
$$
for any $0<\eta <1$ and thus, we have
$$
(u_\eta -1)_+\rightarrow 0 \quad \text{and} \quad (u_\eta)_- \rightarrow 0 \quad \text{in $L^2(\Omega)$ as $\eta\rightarrow 0$}\,.
$$

Hence, we can show that $u_0 (x)\in [0,1]$ a.e. in $\Omega$. We conclude that
\begin{align*}
\begin{cases}
&\xi \in \partial \mathbb{I}_{[0,1]}(u_0)\quad \text{a.e. in $\Omega$,}\\
&v_0 \in \partial \Phi(u_0) \ \ \text{and} \ \ v_0 + \xi + u_0 = v + u \ \text{in $L^2(\Omega)$}\,.
\end{cases}
\end{align*}

This implies that $u_0$ is a minimizer of $\varphi_u$. Then, $u_0 = u$ , and consequently, $v = \xi + v_0$ in $L^2(\Omega)$, which finishes the proof.
\end{proof}

Up to this point, we have worked without imposing any restriction in the dimension of the domain, Now, we introduce the characterization of the subdifferential of the total variation in $L^2(\Omega)$, proposed by Andreu, Ballester, Caselles and Maz\'on in \cite{Mazon} (see also \cite{Mazon_book}), in the specific case of the domain {being} an interval in 1-D, which we take as $(0,1)$ without loose of generality, (see \cite{decico-castra} for a proof).

\begin{thm}[\textbf{TV characterization in 1D}]\label{charactBV}
	Let $u$ be in $BV(0,1)$ such that $\partial \Phi(u) \neq \emptyset$. Then $v\in \partial\Phi(u)$ if and only if there exists $\boldsymbol{z_u}\in H_0^1 (0,1)$ such that a.e. $|\boldsymbol{z_u}|\leq 1$ and
	$$
	 v = -(\boldsymbol{z_u})_x\,, \qquad |Du|=\boldsymbol{z_u}\cdot Du,	
	$$ where the measure $\boldsymbol{z_u}\cdot Du\in\mathcal M(0,1)$ is defined as
\begin{eqnarray*}
(\boldsymbol{z_u}\cdot Du) (U) &:=& \int_U \boldsymbol{z_u}\cdot u_x \,dx+\int_U {\boldsymbol{z_u}}\cdot\nu^u \,d |D^c u| \\
	&+& \sum_{x_\in J_u\cap U}  (u^+(x)-u^-(x)) {\boldsymbol{z_u}}(x)\cdot \nu^u(x),
	\end{eqnarray*} for any Borel set $U\subset (0,1)$.
\end{thm}

With this characterization in mind, the system of Euler--Lagrange equations can be rephrased in the following way:
\begin{prop}\label{prop_E-L}
  Let $(u,v_1,v_2)$ be a minimizer of $F_\varepsilon$. Then, there exist $\boldsymbol{z_u}$, $\boldsymbol{z_{v_1}}$,$\boldsymbol{z_{v_2}}\in H^1_0(0,1)$, {corresponding to $\partial \Phi(u)$, $\partial \Phi(v_1)$ and $\partial\Phi(v_2)$, respectively} as given by Theorem \ref{charactBV}{,} and $g\in\partial \mathbb{I}_{[0,1]}(u)$ such that \begin{equation}\label{descsys}
\begin{cases}
\begin{split}
(\boldsymbol{z_u})_x =\lambda \left( (v_1 - f)^2 - (v_2 -f)^2 \right) + g \\
(\boldsymbol{z_{v_1}})_x= 2\varepsilon\lambda u(v_1 -f) \\
(\boldsymbol{z_{v_2}})_x =2\varepsilon\lambda (1-u)(v_2 -f)\,.
\end{split}
\end{cases}
\end{equation}
\end{prop}

\section{Proofs of the main results}\label{sec:thm}

\subsection{Proof of Theorem \ref{th:constant}}
We need to show first the following auxiliary result.
\begin{lem}[\textbf{Behaviour of $\boldsymbol{z_u}$}]\label{z1}
	Let $u\in BV(0,1)$ and $\boldsymbol{z_u}\in H^1_0(0,1)$, {corresponding to $\partial\Phi(u)$} as provided by Theorem \ref{charactBV}. Then,
	$$
	|\boldsymbol{z_u}| = 1 \,,\quad \text{$|Du|$-a.e}\,.
	$$
\end{lem}
\begin{proof}
	We decompose both measures $|Du|$ and $\boldsymbol{z_u}\cdot Du$ in the following way:
	\begin{align*}
	\qquad \boldsymbol{z_u}\cdot Du &= {(\boldsymbol{z_u}\cdot u_x)} \mathcal L^1 + \boldsymbol{z_u}\cdot D^ju + \boldsymbol{z_u}\cdot D^cu \\
	|Du| &= |u_x|\mathcal L^1 + |{ D^j}u| + |D^cu|
	\end{align*}
	Since both decompositions are mutually singular, we have
	$$
	\boldsymbol{z_u} =
	\begin{cases}
	\quad \displaystyle\frac{u_x}{|u_x|}\,,& \quad \text{$|u_x|\mathcal L^1$-a.e}\\
		\quad \displaystyle\frac{u^+ - u^-}{|u^+ - u^-|}\,,& \quad \text{$|D^j u|$-a.e}\\
		\quad \displaystyle\frac{D^c u}{|D^c u|}\,,& \quad \text{$|D^c u|$-a.e}
	\end{cases}
	$$
	and thus, we have $|\boldsymbol{z_u}| = 1\,,\quad |Du|-a.e$.
\end{proof}

\begin{proof}[Proof of Thm. \ref{th:constant}]
	Firstly, we note that if $(u,v_1,v_2)$ is a minimizer of $F_\varepsilon$, then, it is easy to see that all variables take values in $[0,1]$ a.e in $(0,1)$ as shown in Lemma \ref{lem:truncation} in the Appendix. Suppose that there exists a Borel set $U\subset (0,1)$ such that $Dv_1   (U)\neq 0$. By Lemma \ref{z1}, we know that there is $x_1\in U$ such that $\boldsymbol{z_{v_1   }}(x_1)\in\{-1,+1\}$, {where $\boldsymbol{z_{v_1}}$ is given by Proposition \ref{prop_E-L}.}
	Then, {by $\eqref{descsys}_2$, } we have the next inequality
	\begin{align*}
	\quad 1 = \left| \int_0^{x_1} (\boldsymbol{z_{v_1}})_x   \,dx\right| =& 2\varepsilon\lambda \left| \int_0^{x_1} u(v_1    - f)\,dx  \right| \leq\\
	\leq& 2\varepsilon\lambda \int_0^1 |u||v_1    - f|\,dx \leq 4\varepsilon\lambda
	\end{align*}
	and in consequence, $\varepsilon \geq 1/(4\lambda)$ and thus, a contradiction by hypothesis.

{In order to get that $v_2$ is also constant, we apply the same argument to $\eqref{descsys}_3$}.
\end{proof}

Under the size condition $\varepsilon<\frac{1}{4\lambda}$, which we assume from now on, we integrate the second and third Eqs. in \eqref{descsys} in $(0,1)$ and we obtain
\begin{equation}\label{c1c2}
v_1    = \frac{\displaystyle\int_0^1 u f\,dx }{\displaystyle\int_0^1 u\,dx}\,, \qquad v_2    = \frac{\displaystyle\int_0^1 (1-u) f \,dx}{\displaystyle\int_0^1 (1-u)\,dx},
\end{equation}
in the case that $u$ is not constant with values $0$ or $1$. If $u\equiv 0$ (resp. $u\equiv 1${)} then $v_1$ (resp. $v_2$) can be any constant value in $(0,1)$. We finally note that, being $v_1,v_2$ constants, the energy functional does not depend on $\varepsilon$. {From now on, we will assume the size condition
$\varepsilon<\frac{1}{4\lambda}$.} We rename the constants as $c_1,c_2$ for consistency and {remove} the $\varepsilon$ dependence by letting \begin{align*} F(u,c_1,c_2) &= |Du|((0,1)) + \int_0^1 \mathbb{I}_{[0,1]} (u)\,dx\, +\\&+\lambda \int_0^1 (u(c_1-f)^2+(1-u)(c_2-f)^2)\,dx\,.  \end{align*}

\subsection{Proof of Theorem \ref{th:binary}}

This Section is devoted to prove that the first coordinate of the minimizer is necessarily a binary BV function. { Hereinafter, we assume that the datum $f$ satisfies $(H)$.} The proof will be done in two different steps. First of all, we show that if  $(u,c_1,c_2)$ is a minimizer, then there is a \lq\lq quasi--piecewise constant\rq\rq competitor $\overline u$ with lower energy than (or equal to) the corresponding to $u$. Then, we prove that the competitor cannot be a minimizer in case it is not binary by using the PDE system \eqref{descsys}.
	
\noindent We start by defining our concept of quasi--piecewise constant function.
\begin{defn}\label{uinterm}
	We say that $u\in BV(0,1)$ is a  {quasi--piecewise constant} if there exists a piecewise constant function $u_s$ and an a.e binary function $u_b$ which fulfill that they are not non-0 simultaneously and
	$$
	u(x) = u_s(x) + u_b(x)\,, \quad \text{a.e. $x \in (0,1)$}\,.
	$$
\end{defn}

\vspace{0.1cm}
\begin{thm}\label{existuint}{Let} $c_1,c_2\in [0,1]$.
	Given $u$ a minimizer of $F(\cdot,c_1,c_2)$, there exists a quasi--piecewise constant $\overline u\in BV(0,1)$ satisfying \begin{equation}\label{tostep}
F(\overline u,c_1, c_2) \leq F (u, c_1, c_2)\,.
\end{equation}
\end{thm}
\begin{proof}
	Let $x_0$ be in $(0,1)$ such that $u(x_0)\notin \{0,1\}$ and $x_0\notin J_{u}$, and let $\boldsymbol{z_{u}}\in H_0^1(0,1)$ be the {vector field associated to $\partial\Phi(u)$ as given }by Theorem \ref{charactBV}{.}
	We distinguish two cases:
	\begin{enumerate}[(i)]
	\item The case $x_0\in(|\boldsymbol{z_{u}}|)^{-1}(\{1\})$: We assume without loss of generality that $\boldsymbol{z_{u}}(x_0)=1$  (for $\boldsymbol{z_{u}}(x_0)=-1$, the argument is analogous).
	Under this assumption, we know that ${u}$ is continuous at $x_0$ {(remember that we always identify a BV function with its precise representative)} and thus, ${u} \notin \{0,1\}$ {in} a neighbourhood of $x_0$ denoted by $\mathcal{E}_{x_0}$. Hence, using {$\eqref{descsys}_1$}, we know  that
	\begin{equation}\label{thm2:inter1}
	(\boldsymbol{z_{u}})_x = \lambda\left((c_1 - f)^2 - (c_2 - f)^2\right) \qquad \text{in $\hspace{0.1cm}\mathcal{E}_{x_0}$}\,.
	\end{equation}
	Since $f\in BV(0,1)$, we have by the above expression that $(\boldsymbol{z_{u}})_x\in BV(\mathcal{E}_{x_0})$  and thus, its lateral traces $(\boldsymbol{z_{u}})_x^+(x_0)$ and $(\boldsymbol{z_{u}})_x^-(x_0)$ are well defined. In addition, since  $\boldsymbol{z_u}(x_0)=1$ and $\|\boldsymbol{z_u}\|_{{L^\infty(\Omega)}} \leq1$, this implies that
	$$
	\nonumber(\boldsymbol{z_{u}})_x^+ (x_0)\leq 0 \leq (\boldsymbol{z_{u}})_x^- (x_0)
	$$
	and we have, by \eqref{thm2:inter1} that
	$$
	f^-(x_0) \leq \frac{c_1 + c_2}{2} \leq f^+(x_0)\,.
	$$
	Then, 
we note that the set
	\begin{align*}
	&\left\{ x\in(0,1)\backslash(J_u\cup J_f)  : |\boldsymbol{z_{u}}(x)| = 1 \wedge {u}(x)\in (0,1)\hspace{0.1cm} \right\}
	\intertext{is a subset of the following set:}
	&A:=\left\{ x\in (0,1)\backslash  J_f : {u}(x)\in(0,1) \hspace{0.1cm}\wedge\hspace{0.1cm} f(x) = \frac{c_1 + c_2}{2}  \right\}\,.
	\end{align*}
	We note that $A = A^o \cup (\partial A\cap A)$ where $A^o$ and $\partial A$ are the (topological) interior and boundary of $A$, respectively. Particularly, we note that
	$$
	A^o = \bigcup_{k=1} I_k,
	$$
		where $\{I_k = (a_k, b_k)\}_k$ {is a disjoint collection of open intervals}; and $\mathcal{L}^1(\partial A\cap A) = 0$ by assumption on $f$ and the fact that $\partial A\cap A\subset J_f\cup \partial\{f=\frac{c_1+c_2}{2}\}$.

\smallskip

Next, we will modify $u$ in each interval $I_k$ to decrease the energy. Before doing it, we point out that, for $c_1,c_2$ fixed, minimizing $F$ is equivalent to minimize \begin{align*}G(w)&:=|Dw|((0,1)) +\int_0^1 \mathbb I_{[0,1]}(w)\,dx \,  + \\ &+\lambda\int_0^1 w(c_1-c_2)(c_1+c_2-2f)\,dx\ .\end{align*}

{ Since $I_k\subset A$, we observe that  $$G(u\chi_{I_k})=|Du|((a_k,b_k)).$$} Therefore, if we take $u_k:=u(a_k)^+\chi_{I_k}+ u\chi_{(0,1)\setminus I_k}$ {it is easy to show that} $$F(u_k,c_1,c_2)\leq F(u,c_1,c_2).$$



\item If $x_0\notin(|\boldsymbol{z_{u}}|)^{-1}(\{1\})$: Since $\boldsymbol{z_{u}}\in H^1_0(0,1)\subset C(0,1)$, we take the largest interval $I$ {containing} $x_0$ such that $|\boldsymbol{z_{u}}|(I)\subseteq [0,1)$. By Lemma \ref{z1}, $|Du|=0$ in this interval and thus, $u$ is constant in $I$.
	\end{enumerate}

	\vspace{0.1cm}
	\noindent Consequently, we note that $(0,1)\backslash (\partial A\cap A)$ {can} be decomposed as:
	$$
	\underset{\displaystyle B}{\underbrace{\left((0,1)\cap ({u})^{-1}(\{0,1\}) \right)}} \cup \underset{\displaystyle C}{\underbrace{\left((0,1)\cap ({u})^{-1}((0,1)) \right)}}
	$$
	such that $C$ is a subset of
	$$
	\left(\bigcup_{k=1} I_k\right) \cup \left(C\cap (|\boldsymbol{z_{u}}|)^{-1}([0,1)) \right)\,.
	$$
	
	\vspace{0.1cm}
	\noindent Defining
	$$
	\overline u(x):=\begin{cases}
	\quad {u_k(x)} \quad & \text{if $x\in I_k$ for any $k$}\\
	\quad u(x) \quad & \text{if $x\notin I_k$ for all $k$}\,,
	\end{cases}
	$$
	we obtain that the inequality \eqref{tostep} is {clearly satisfied.} Note that $\overline u$ is a piecewise constant function in $C$ because of the reasoning in (i) and (ii) and thus, {we can take $$u_s:= \overline u\mathcal{X}_{C}.$$ Finally, }as $B$ and $C$ are disjoint, we obtain that $\overline u$ is a quasi--piecewise constant function.
\end{proof}

\bigskip
We introduce now two useful remarks:
\begin{rem}\label{useremark1}
	Let $(u,c_1^u,c_2^u)$ be a minimizer of $F$ such that $c_1^u \leq c_2^u$. Defining $w := 1 - u$, it is easy to show that $ c_1^w = c_2^u$, $c_2^w = c_1^u$ and that
	 $F(w,c_1^w, c_2^w) = F(u,c_1^u, c_2^u)$. On account of it, we assume hereinafter that $c_2\leq c_1$.
	Furthermore, if $c_1^u=c_2^u$, we note that
	$$
	F(u, c_1^u,c_2^u) = |Du|((0,1)) + \lambda\int_0^1 (c_2^u - f)^2\,dx
	$$
	and thus, in order to be a minimizer, $u$ is a necessarily a constant function.
\end{rem}
\begin{rem}\label{useremark2}
	Let $(u,c_1,c_2)$ be a minimizer of $F$ and suppose that there exist $a<b\in J_u$ such that $u^+(a)$, $u^-(b)$, $u(x)\notin \{0,1\}$, for any $x\in (a,b)$. {By integrating { $\eqref{descsys}_1$} on} $(a,b)$, we have that
	$$\label{jumpkey}
	\lambda\int_{a}^{b}\left( (c_1 - f)^2 - (c_2 - f)^2\right)\,dx =  \boldsymbol{z_{u}}(b) - \boldsymbol{z_{u}}(a),
	$$
{where $\boldsymbol{z_{u}}$ corresponds to $\partial\Phi(u)$}.
	Note that {the right hand side} term is equal to $2$, $-2$ or $0$ because of {the fact that} $|\boldsymbol{z_{u}}|=1$ in the jump set $J_{u}$ as a consequence of Lemma \ref{z1}; as $\boldsymbol{z_{u}}\in H^1_0(0,1)$, if $x_0\in J_{u}$, then
	\begin{itemize}
		\item $\boldsymbol{z_{u}}(x_0) = -1$ when $u$ jumps toward a lower step in $x_0$.
		\item $\boldsymbol{z_{u}}(x_0) = 1$ when $u$ jumps toward an upper step in $x_0$.
	\end{itemize}
\end{rem}
After these remarks, we prove the following statement:

\begin{prop}\label{uconstandbin}
	Let $(u,c_1,c_2)$ be a minimizer of $F$ such that $u$ is a quasi--piecewise constant function. Then either $u$ is constant or $u$ is an a.e. binary function.
\end{prop}
\begin{proof}
	This statement is proved by contradiction. We suppose that $u$ is not {an a.e} binary function and not constant, i.e, $u_s$ has some non-binary step ($u_s$ being the piecewise constant part). Let {$a,b\in J_u$ be such that} $u_s((a,b)) = \{\beta\}\notin\{0,1\}$. In addition, let $\alpha:=u^-(a)$ and  $\gamma:=u^+(b)$.
	{According to the relation $\alpha \neq \beta \neq \gamma$, there are three different cases to study:}
	\begin{enumerate}[(i)]
		\item $\hspace{0.1cm}\alpha < \beta < \gamma$ (or $\gamma < \beta < \alpha$, {resp.}):
		We define
		$$
		v_\tau(x) := \begin{cases}
		u(x) \quad &\text{ if  $x \notin (a,b)$ }\\
		\tau \quad &\text{ if $x\in(a,b)$ }
		\end{cases},
		$$
		where $\tau\in(\alpha,\beta)$ (or $\tau\in (\gamma, \beta)$, {resp.}). In any case, according to Remarks \ref{useremark1} and \ref{useremark2}, we can suppose that $c_1>c_2$ and
		$$
		\lambda\int_a^b \left((c_1 - f)^2 - (c_2 - f)^2\right)\,dx \,{= \lambda(c_1 - c_2) \int_a^b \left(c_1 +c_2 - 2f\right)dx} = 0\,.
		$$
		{ 
			Since $\lambda(c_1-c_2)>0$, we obtain}
		$$
		\int_a^b  f \,dx = \left(\frac{c_1 + c_2}{2}\right)\left(b-a\right).
		$$
		It is clear that $F(u,c_1, c_2) = F({ v_\tau}, c_1 , c_2)$. Then, $(v_\tau, c_1,c_2)$ is a minimizer too. { Consequently, by  \eqref{c1c2}, we have } 
		\begin{align*}
		c_1 = \frac{\displaystyle\int_0^1 uf \,dx\hspace{0.1cm}}{\displaystyle\int_0^1 u\,dx \hspace{0.1cm}} &= \frac{\displaystyle\int_0^1 v_\tau { f} \,dx\hspace{0.1cm}}{\displaystyle\int_0^1 v_\tau\,dx \hspace{0.1cm}}\qquad &\text{i.e.}\\
		\frac{\displaystyle \int_I uf{\,dx \hspace{0.1cm} + \beta ML}}{\displaystyle \int_{I} u{ \,dx \hspace{0.1cm} + \beta L}} &= \frac{\displaystyle \int_I u f {\,dx \hspace{0.1cm} + \tau ML} }{\displaystyle \int_I u {\,dx \hspace{0.1cm} + \tau L}}\,,&
		\end{align*} where we denoted by $I:=[0,a)\cup (b,1]$, $M:=\frac{c_1+c_2}{2}$ and $L:=b-a$. { Then,
		{\begin{align*}\label{eq:contrc1}
			\beta M L \int_I u\,dx  + \tau L \int_I u{ f}\,dx \hspace{0.1cm} &=\tau M L \int_I u\,dx  + \beta L \int_I u{{f}}\,dx\\
			\intertext{Since $\beta\neq \tau$ and $L \neq 0$, we have}
				M  \int_I u\,dx \hspace{0.1cm} &= \int_I u{f}\,dx\,.
			\end{align*}}}
  {This yields
{$$
		M \left( \int_0^1 u \hspace{0.1cm}  - \beta L  \right) = \int_0^1 uf \hspace{0.1cm}  - \beta M L  \quad \text{i.e}
		\ \ M = \frac{\displaystyle\int_0^1 uf\,dx \hspace{0.1cm}}{\displaystyle\int_0^1 u \,dx\hspace{0.1cm} } = c_1\,,
$$}}
		thus leading to $c_1 = c_2$, i.e, to a contradiction {by Remark \ref{useremark1}.}

\smallskip

		\item $\hspace{0.1cm} \beta < \alpha {\leq} \gamma$ (or $\beta < \gamma { \leq} \alpha$, {resp.}):
		Similarly as before, we consider
		$$
		v_\tau(x) := \begin{cases}
		u(x) \quad &\text{ if  $x \notin (a,b)$ }\\
		\tau \quad &\text{ if $x\in(a,b)$ }\,,
		\end{cases}
		$$
		where $\tau = \alpha$ (or $\tau = \gamma$, {resp.}). In any case, according to Remarks \ref{useremark1} and \ref{useremark2}, we can suppose that $c_1>c_2$ and
		$$
		\lambda \int_a^b \left((c_1 - f)^2 - (c_2 - f)^2\right)\,dx = 2.
		$$
		{We observe that}
		$$
		\int_a^b {f}\,dx = ML - A, 
		$$
		where $M$ and $L$ are the constants defined in the previous case and $A:=\frac{1}{\lambda(c_1 - c_2)}$. Then, it is easy to check that $F(u,c_1, c_2) = F(v_\tau,c_1,c_2)$. Again, $(v_\tau, c_1,c_2)$ is a minimizer and thus, it satisfies \eqref{c1c2}. Repeating the reasoning in the previous case (replacing $ML - A$ instead of $ML$), we obtain
		$$
		M - \frac{A}{L}= c_1
		$$
		and if we redo the same computations for $c_2$ we obtain the same equation, thus leading to $c_1 = c_2$, i.e, to a contradiction {as before.}

\smallskip

		\item $\hspace{0.1cm} \alpha { \leq} \gamma < \beta$ (or $ \gamma { \leq} \alpha < \beta$, {resp.}):
		In this case, we define
		$$
		v_\tau(x) := \begin{cases}
		u(x) \quad &\text{ if  $x \notin (a,b)$ }\\
		\tau \quad &\text{ if $x\in(a,b)$ }\,,
		\end{cases}
		$$
		where $\tau = \gamma$ (or $\tau = \alpha$, {resp.}). Once again, repeating the computations in the previous case, we end up with $c_1=c_2$, thus finishing the proof.
	\end{enumerate}
\end{proof}
We finish this Section by pointing out that Theorem \ref{existuint} and Proposition \ref{uconstandbin} already {prove} Theorem \ref{th:binary}.

\section{Properties of minimizers}\label{sec:prop}

\subsection{Proof of Properties $(a)$ and $(b)$}

In this Section, we show that the set of minimizers of $F$ coincides with that of minimizers to Chan-Vese and we prove Properties $(a)$ and $(b)$ of the minimizers.

\begin{rem}
    It is obvious that, given $E\subset \Omega$ of finite perimeter,  \begin{equation*}\begin{split}\label{chan-vese=ours1}F(\chi_E,c_1,c_2)&= {\rm Per}(E;(0,1)){ + }
\lambda \int_E (c_1-f)^2\,dx+ { \lambda}\int_{(0,1)\setminus E}(c_2-f)^2\,dx.\end{split}\end{equation*} Then, we note that \begin{eqnarray}\label{chan-vese=ours}\nonumber \min_{u\in L^2(0,1),c_1,c_2} F(u,c_1,c_2) &\leq& \min_{E,c_1,c_2} {\rm Per} (E;\Omega)+ {\lambda} \int_E (c_1-f)^2\, dx\\ && + { \lambda}\int_{\Omega\setminus E} (c_2-f)^2\,dx.\end{eqnarray} On the other hand, by Theorem \ref{th:binary}, we now that the minimum of $F$ is achieved in a constant or binary BV function $u$, for the first coordinate. Then, $u=\chi_E$ (or $u=k\in [0,1]$)for a set $E\subseteq (0,1)$ of finite perimeter, which proves the reverse inequality in \eqref{chan-vese=ours}. This shows that minimizers of $F$ coincide with Chan-Vese's minimizers.
\end{rem}

We now prove Properties $(a)$ and $(b)$:

\noindent Suppose that $x_0\in J_u$ and that $u^-(x_0)=1$, $u^+(x_0)=0$. Then, as explained in Remark {\ref{useremark2}}, ${\boldsymbol z_u}(x_0)=-1$, with ${\boldsymbol z_u}$ {corresponding to $\partial\Phi(u)$} as given  by Theorem \ref{charactBV}. Moreover, since  $\partial\mathbb I_{[0,1]}(x)=]-\infty,0]\delta_{0}+[
0,+\infty[\delta_{1}$, by $\eqref{descsys}_1$, as in the {Proof of Theorem \ref{existuint}} that
$$
f^+(x_0)\leq \frac{c_1 + c_2}{2}\leq f^-(x_0).
$$
If instead $u^-(x_0)=0$, $u^+(x_0)=1$, one gets
$$
f^+(x_0)\geq \frac{c_1 + c_2}{2} \geq f^-(x_0).
$$
Therefore \begin{equation}\label{jumplevelset} J_u\subseteq C_{1,2}:=\{x\in (0,1) : f(x)\ni\frac{c_1+c_2}{2}\},\end{equation} where the image of a jump point is understood in a multivalued sense (i.e. $f(x)=[\min\{f^-(x),f^+(x)\},\max\{f^-(x),f^+(x)\}]$). Note that this fact almost proves Properties $(a)$ and $(b)$. The only remaining thing to prove is that a jump point cannot exist in the interior of $C_{1,2}$. Suppose, by contradiction that $x_0\in J_u\cap C_{1,2}^o$ and let $a<x_0$ be such that $J_u\cap [a,x_0[=\emptyset$. Without loosing generality, we can suppose that $u=1$ in $[a,x_0[$ and that $u^+(x_0)=0$. Then, considering $v:=u\chi_{(0,1)\setminus [a,x_0[}$, we easily get that $F(v,c_1,c_2)=F(u,c_1,c_2)$ and, therefore, $(v,c_1,c_2)$ is a minimizer too. Then, repeating the reasoning in Proposition \ref{uconstandbin}, in case $(ii)$ we arrive at a contradiction, which finishes the proof of Properties $(a)$ and $(b)$.\\

\begin{rem}
  We note that Properties $(a)$ and $(b)$ of the minimizers are only expected to hold in the one--dimensional case. In fact, an easy counterexample in two dimensions is provided by $f=\chi_E$ with $E=[-{\frac{1}{2},\frac{1}{2}}]^2$ and $\Omega=[-1,1]^2$. In this case $J_f$ is precisely the boundary of the square $[-{ \frac{1}{2},\frac{1}{2}}]^2$ while it can be proved that, for any { $\lambda>\frac{16}{3}$}  Chan-Vese's minimizer cannot be either $u=\chi_\Omega$ or $u=\chi_E$ 
  , thus showing that $(a)$ and $(b)$ do not hold.\\
  
   In fact, Chan-Vese's energy for these two candidates is exactly $\frac{3}{4}\lambda$ {(for $u=\chi_\Omega$, { $c_1=\frac{1}{4}$}) and $4$ (for $u=\chi_E$, { $c_1=1$, $c_2=0$}). Therefore, for { $\lambda>\frac{16}{3}$}, $u=\chi_\Omega$ is not a minimizer. On the other hand, it is easy to modify the corners of the square by reducing the perimeter and not changing to much the {fidelity} term in the energy. We modify $E$, calling the new set $E_\delta$, by removing the corners through the use of $\delta$-radius circumferential arches tangent to every two contiguous sides of $\partial E$. (see Fig. \ref{fig:vdelta})}. 
  Then one can check that $v_\delta:= \chi_{E_\delta}$, { $c_1=1$, $c_2=\frac{(4-\pi)\delta^2}{3+(4-\pi)\delta^2}$}  for $\delta$ satisfying $$\frac{3\delta}{3+{ (4-\pi)}\delta^2}<\frac{2}{\lambda},$$ has strictly less energy. This fact is related to the non-calibrability of the set $E$ with respect to the isotropic norm in the total variation (see \cite{alter-caselles-chambolle}).\\
  
  \begin{figure}[h!]
	\centering
	\includegraphics[width=.8\linewidth]{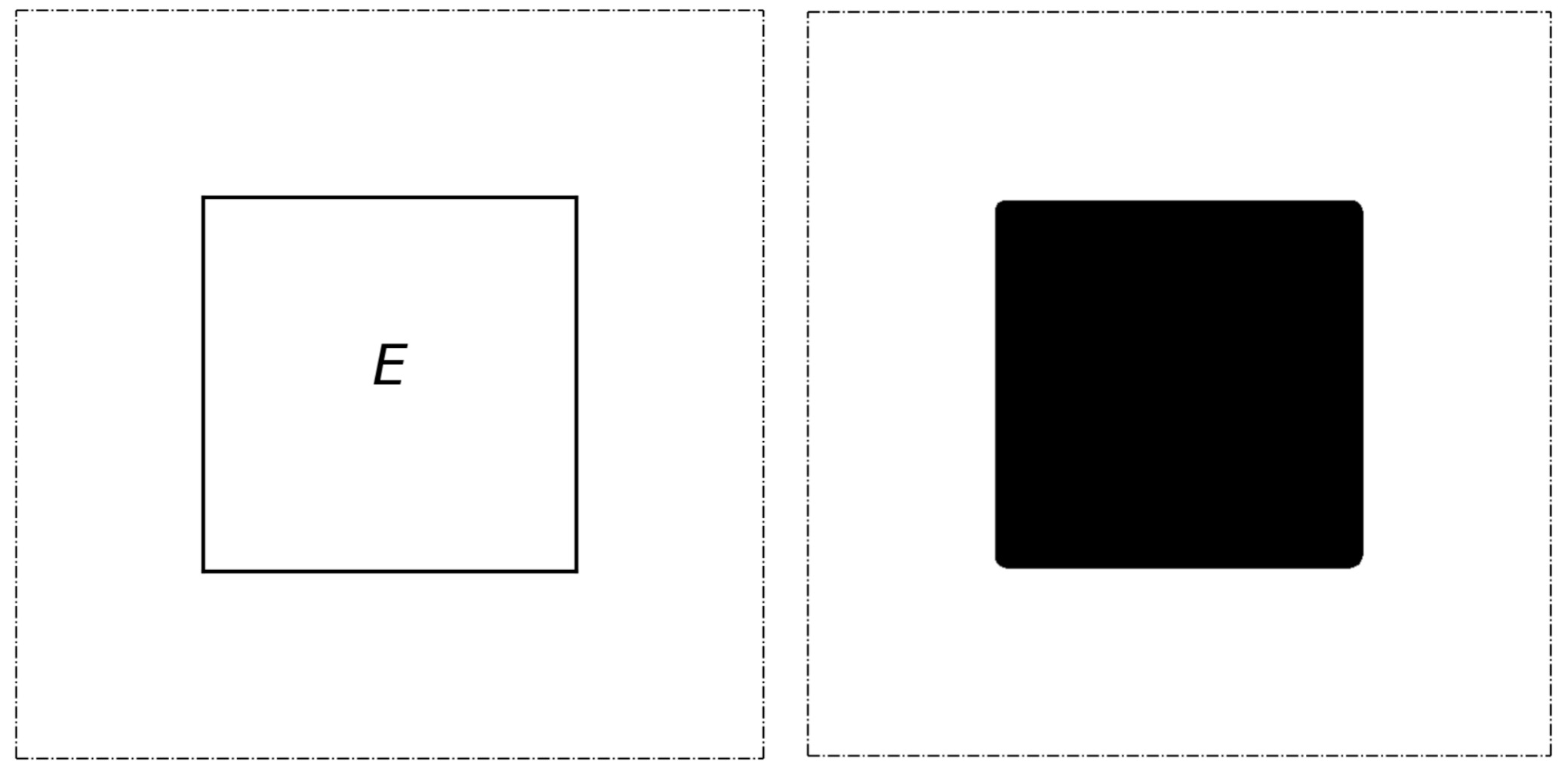}
	\caption{\textit{Left}: $\partial E$ and $\partial \Omega$ represented by solid and dashed lines, respectively. \textit{Right}: Chan-Vese segmentation, where black and white colours represent 1 and 0 numerical values, respectively.}
	\label{fig:propertyfail}
\end{figure}
\begin{figure}[h!]
	\centering
	\includegraphics[width=.4\linewidth]{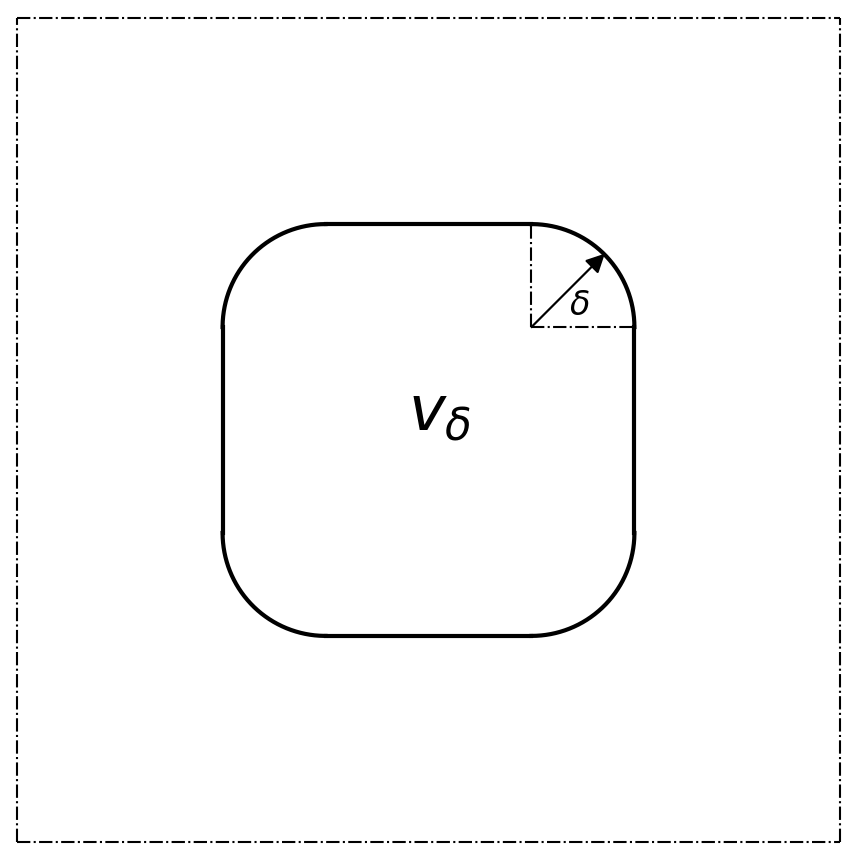}
	\caption{$\partial E_\delta$, i.e, $J_{v_\delta}$.}
	\label{fig:vdelta}
  \end{figure}

  Therefore, if one wants to obtain similar results to Properties $(a)$ and $(b)$ in higher dimensions, the total variation term needs to be changed by an anisotropic version of it as in the case of the anisotropic Rudin--Osher--Fatemi functional, for which stability of piecewise constant functions on rectangles has been recently shown in \cite{lasica-moll-mucha} and \cite{Kirisits-Setterqvist-Scherzer}. We will investigate this issue further in a subsequent paper.

\end{rem}

\subsection{Application of the properties}
\label{sec:algorithm}

In this Section, we  {propose} a trivial way to approach the solution of the 1D Chan-Vese problem 
by using properties $(a)$ and $(b)$ of the minimizer. {{We will also comment on} the pros of this trivial algorithm in front of those based on a Gradient Descent (GD) scheme. {We} remark that those algorithms are applied in an 1-D version of alternating scheme proposed by Chan and Vese in \cite{Chan-Vese}. Hereinafter, we will assume that the boundary of each level of the datum $f$ set has a finite number of points.}

\medskip
{We start with the general case.  In this} one, the idea is:
\begin{enumerate}
	\item[(1)] Take a discretization of the range, {thus} defining the working level sets.
	\item[(2)] In each level set, compute the binary candidate with the least energy {with jumps in the boundary of the level set}.
	\item[(3)] Compare between the solutions and choose one with the {smallest} energy.\\
\end{enumerate}

{Besides}, in the case of $f$ being a step function, we can further simplify the previous idea thanks to the implementation of the inclusion $J_u \subseteq J_f$. This reduction is based on trying all the possible combinations of characteristic solutions whose jump set is a subset of $J_f$.

\begin{figure}[h!]
	\centering
	\includegraphics[width=.5\linewidth]{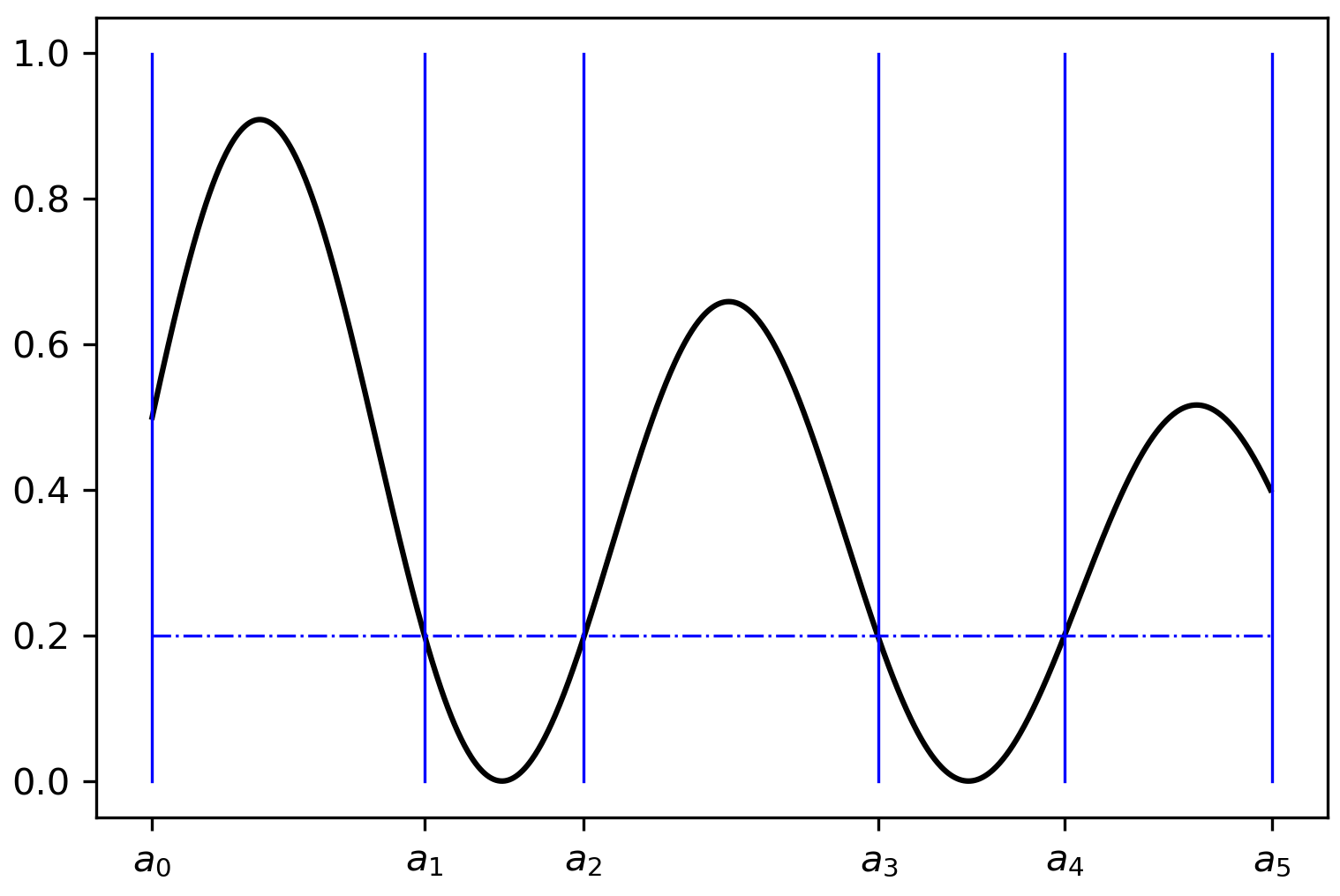}
	\caption{{Given the level set (illustrated by the dashed blue line), the candidates take a constant ($0$ or $1$) value between each couple of vertical lines}.}
	\label{fig:algorithm}
\end{figure}

In Figure \ref{fig:algorithm}, {we sketch how to perform Step 2 in the general algorithm for a fixed level set}.  {Firstly, we obtain all possible jump points for the possible minimizer $u$, ($a_i , i\in\{1,...,m\}$). Then, we know that the candidate to minimizer takes the form} {$$u=\chi_{\cup_{j\in\mathcal{I}}[a_j,a_{j+1}]} \quad \text{where}\quad \mathcal{I}\subseteq\{0, ...,m-1\}$$}

Computing all the {possibilities} we keep the {candidate} with the least energy among them. {We compare its energy with that of the candidate obtained from a previous bigger level set and the procedure is repeated until we reach the lowest level set in the discretization}. To show the suitability of this trivial approach in some situations, we present the following example:

\begin{ex}
	Suppose we segment by the 1-D Chan-Vese model a signal whose shape is similar to the one of the Weierstrass function. This kind of signals has the property of exhibiting abrupt variations of its slope on the whole domain, which causes problems for GD based algorithms. {This intuitive idea can be seen in Figure \ref{fig:weierstrass}, where we {compute an approximation to} the minimizer using two {different} approaches: {in} one {of them we use} the alternating Chan-Vese scheme with a GD-based method (ADAGRAD algorithm, see \cite{adagrad}); {while on the other one we use} the trivial scheme explained above}.

\begin{figure}[h!]
\begin{subfigure}{.49\textwidth}
	\centering
	\includegraphics[width=1\linewidth]{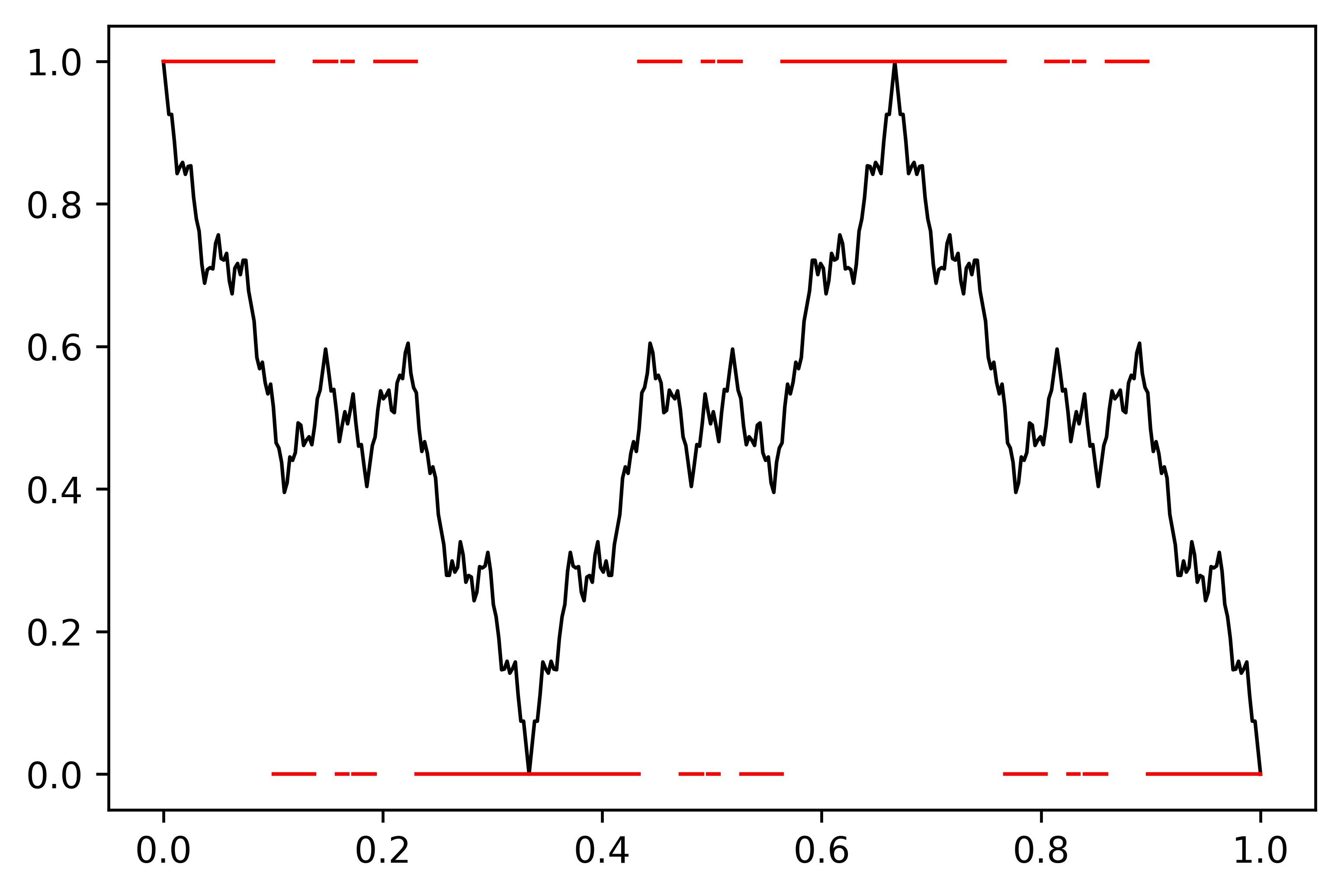}
\end{subfigure}
\begin{subfigure}{.49\textwidth}
	\centering
	\includegraphics[width=1\linewidth]{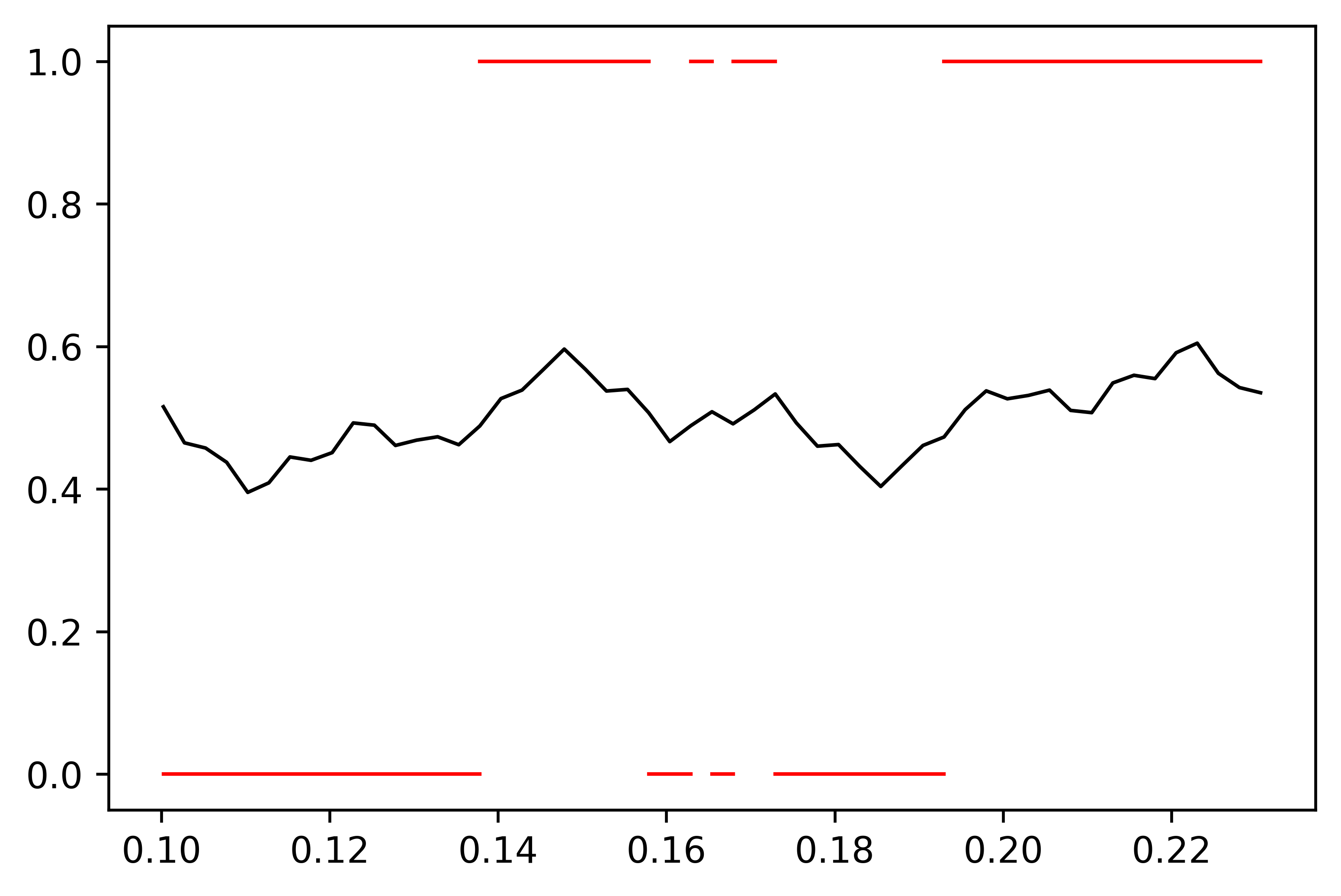}
\end{subfigure}
\begin{subfigure}{.49\textwidth}
	\centering
	\includegraphics[width=1\linewidth]{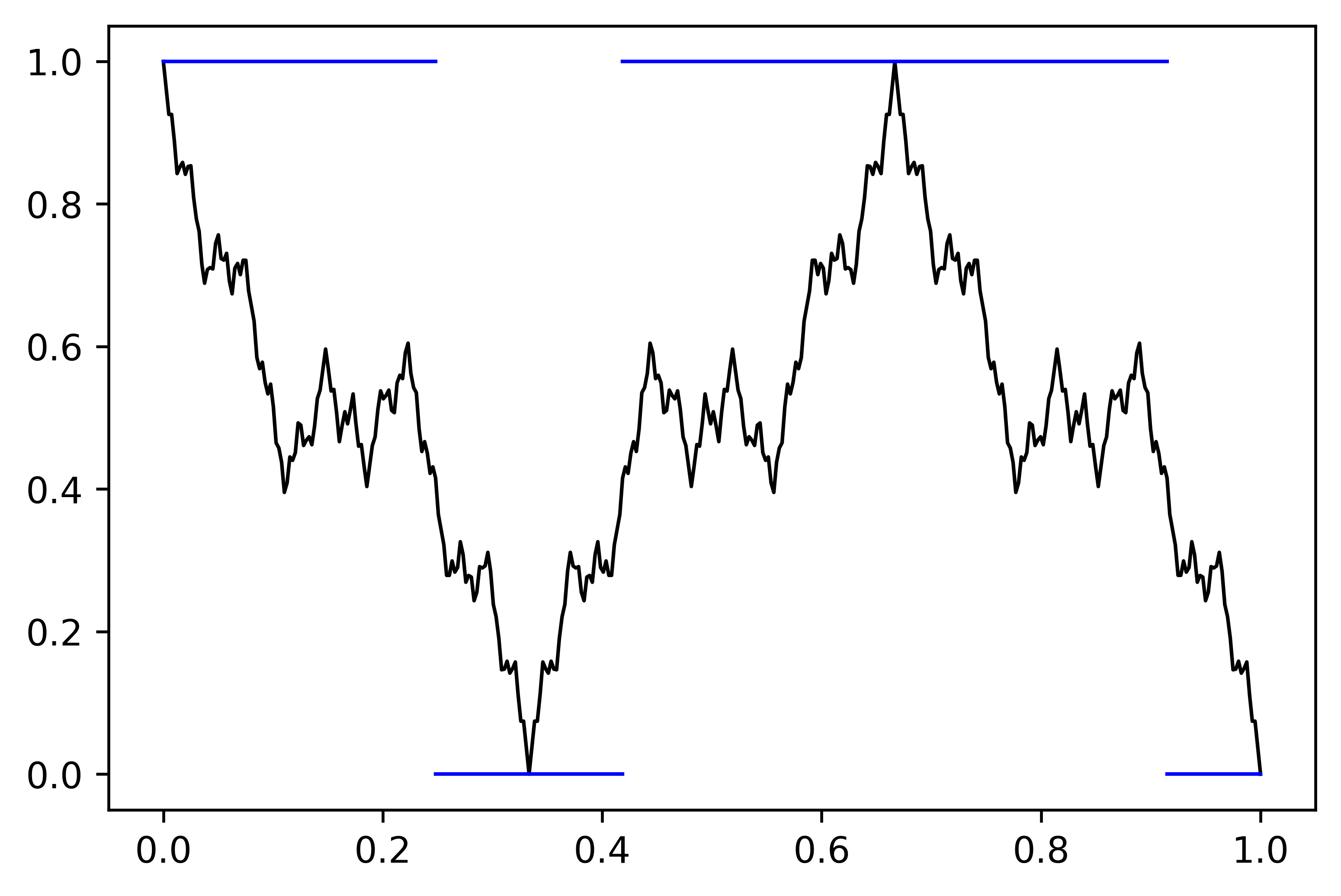}
\end{subfigure}
\begin{subfigure}{.49\textwidth}
	\centering
	\includegraphics[width=1\linewidth]{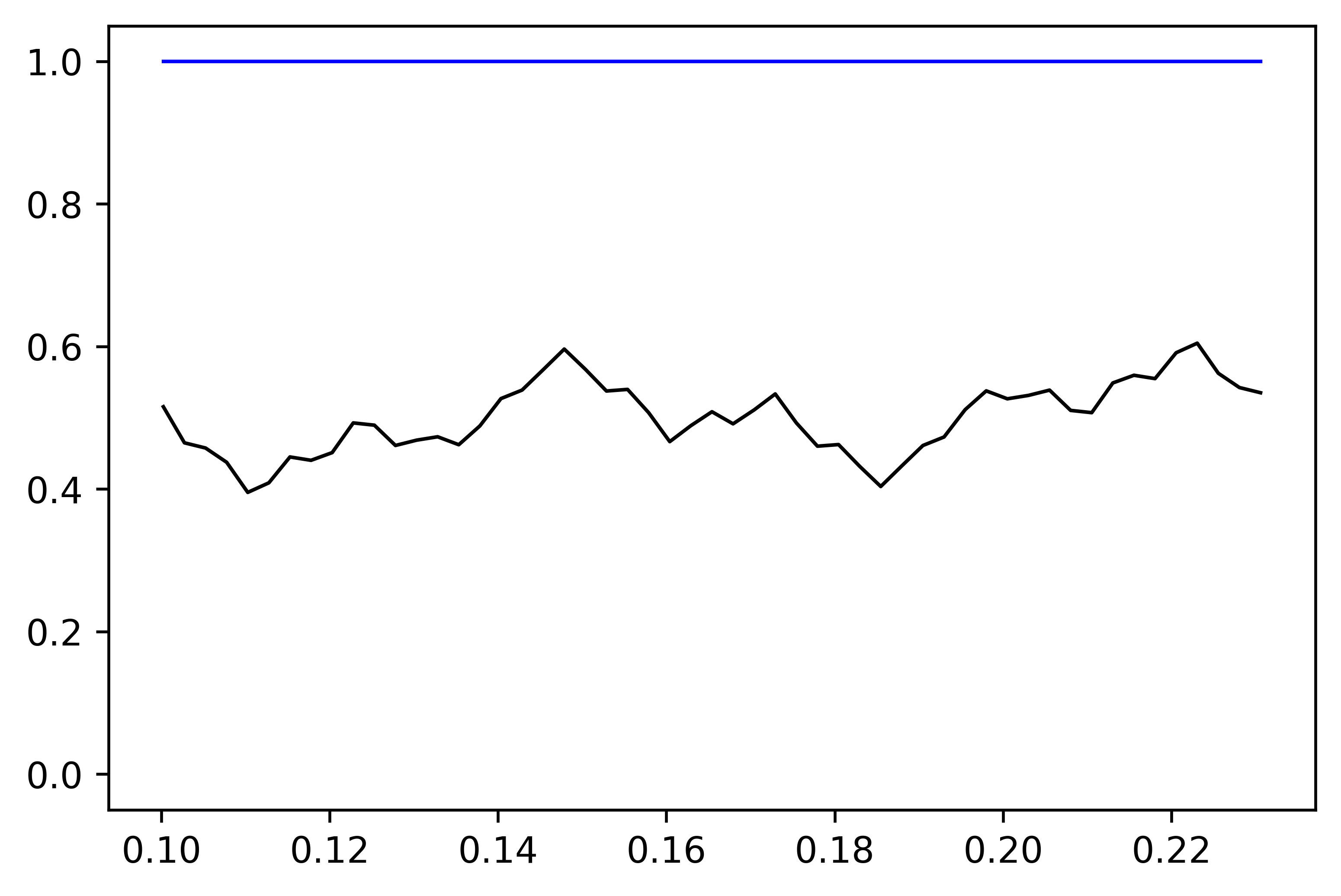}
\end{subfigure}
	\caption{Comparison of different approaches in the segmentation of a Weierstrass type function by 1-D Chan-Vese {(original signal in black; in colours the result of the segmentation).}
	\textit{Top:} GD based approach.  \textit{Bottom:} Trivial approach.
	\textit{Left:} Result in $[0,1]$.  \textit{Right:} Result {zoomed} in $[0.1, 0.23]$.}
	\label{fig:weierstrass}
\end{figure}
\end{ex}

{We note how the GD based approach suffers from the variations of the signal, {thus} affecting to the performance of the alternating Chan-Vese scheme. In contrast, the trivial approach, based on properties $(a)$ and $(b)$, provides an adequate minimizer approximation. Moreover, {we} note that even {if we chose} a particular case of {an} algorithm based on GD, this type of behaviour is a trend in {any of} these algorithms. Therefore, the use of the scheme presented in this paper is beneficial in situations where the application of GD (or variants) gives a wrong segmentation of the signal, as shown in the figure above.}

\appendix

\section{Appendix}
In this Appendix, we show that the functional $F_\varepsilon$ always  has  a minimizer regardless of the dimension of the domain. Note that, for $\varepsilon\leq {\frac{1}{4\lambda}}$, in the $1$--dimensional case, according to Remark 3,  existence of minimizers follows directly from existence of minimizers to Chan-Vese's functional. Here, we give a direct proof for any $\varepsilon>0$, which, in turn, provides an alternative proof of existence of solutions to \eqref{chan-vese}. Existence of minimizers will be shown through the study of the following auxiliary energy: $G_\varepsilon: (L^2(0,1))^3\to [0,\infty]$,
\begin{equation*}
 G_\varepsilon (u,v_1, v_2):=  F_\varepsilon (u,v_1, v_2) + \int_\Omega \left(\mathbb{I}_{[0,1]}(v_1) + \mathbb{I}_{[0,1]}(v_2)\right)\,dx\,.
\end{equation*}

Existence of minimizers to functional $G_\varepsilon$ is guaranteed by the Direct Method in the Calculus of Variations since the functional is easily seen to be lower semicontinuous in $L^2(\Omega)$ and coercive. Now, we relate both functionals $F_{\varepsilon}$ and $G_\varepsilon$ in the following result:

\begin{lem}\label{lem:truncation}
	Let $\left(u,v_1, v_2\right)\in \left(L^2(\Omega)\right)^3$. Then,
	\begin{equation}\label{FwithT}
	F_\varepsilon\left(u,T\left(v_1\right),T\left( v_2\right)\right)\leq  F_\varepsilon\left(u,v_1,v_2\right)\,,
	\end{equation}
	where $T$ is the truncation function in $[0,1]$; i.e. $T(t):=(t-1)_+-t_-$.
\end{lem}

\begin{proof}
	Firstly note that we can suppose that $u(\Omega)\in [0, 1]$ a.e.  (otherwise, both terms in the inequality are equal to $+\infty$). Now, we observe that the following inequality is easily seen to be true (since $f(\Omega)\in [0,1]$):
		\begin{equation*}
		 			  (T(v_i))-f)^2\leq (v_i -{f})^2\,, \qquad \forall i\in {1,2}.
	\end{equation*}
	In consequence,
	\begin{align}\label{Tv}
	\nonumber &\int_\Omega \left(u\left(T(v_1)-{f}\right)^2 + (1-u)\left(T(v_2)-{f}\right)^2\right)\,dx \leq\\
	\leq &\int_\Omega \left(u\left(v_1-f\right)^2 + (1-u)\left(v_2-{f}\right)^2\right)\,dx\,.
	\end{align}
	
	Moreover, applying the Chain Rule in the composition of a $BV$ function and a Lipschitz function ({\cite[Theorem 101]{Ambrosio}}), we have
	\vspace{0.25cm}
	\begin{equation}\label{chain rule}
	 \left|D\left(T\left(v_i\right)\right)\right|(\Omega) \leq \left|Dv_i\right| (\Omega)\,, \qquad \forall i\in\{1,2\}\,.
	\end{equation}
	Therefore, by \eqref{Tv} and \eqref{chain rule} we conclude \eqref{FwithT}.
\end{proof}

\medskip
{ \begin{prop} For any $\varepsilon>0$, the functional $F_\varepsilon$ has a minimizer $(u,v_1,v_2)\in BV(\Omega)^3$.
\end{prop}}
\begin{proof}{ The proof follows directly from Lemma \ref{lem:truncation} since} we obtain that
	\begin{equation*}
	\text{$\left(u,v_1, v_2\right)$ is a minimizer of $F_\varepsilon$ $\longleftrightarrow \quad $ it is a minimizer of $G_\varepsilon$}\,.
	\end{equation*}
%

Therefore, by the existence of minimizers to $G_\varepsilon$ we can conclude that $F_\varepsilon$ admits (at least) one minimizer (which moreover belongs to $(BV(\Omega))^3$).
\end{proof}

\medskip

We finish this Appendix with the proof of the

\smallskip

{\noindent \bf{Claim}:} Let $h:\R\to \R$ be a Lipschitz non decreasing function. Then, the following { equality} holds:
  \begin{equation*}
    \langle v, h(u)\rangle= |Dh(u)|(\Omega),\quad
  \forall v\in \partial\Phi(u).\end{equation*}

\begin{proof}
  Since this Claim is stated in the multidimensional case, we need to recall (see \cite{Mazon}) that $v\in\partial\Phi(u)$ if, and only if, there exists a vector field $\z$ with $||\z||_{L^\infty(\Omega)}\leq 1$, ${\rm div }{\z}\in L^2(\Omega)$ such that $v=-{\rm div} \z$, $[\z,\nu^\Omega]=0$ and $$\langle v, u\rangle=|D u|(\Omega)=\int_\Omega d(\z,Du)=\int_\Omega \theta(\z,Du)\,d|Du|,$$ with $(\z,Du)$ being the Radon measure defined by (see \cite{Anzelloti} for precise definitions and results here stated)$$(\z,Du)(\varphi)=-\int_\Omega u\varphi {\rm div} \z\,dx-\int_\Omega u \z\cdot \nabla\varphi\,dx\,,\quad {\rm for \ }\varphi\in \mathcal D(\Omega),$$ $\theta(\z,Du)$ being the Radon--{Nikodym} derivative of $(\z,Du)$ over $|Du|$ and $[\z,{\nu^\Omega}]$ being the weak normal trace of $\z$ on the boundary. Note that this implies, in particular, that \begin{equation}
    \label{zDu}\theta(\z,Du)=1\,,\quad |Du|-{\rm a.e.}
  \end{equation}

  Since $h$ is Lipschitz, by Chain's rule, we have that $h(u)\in BV(\Omega)$. Suppose now that $h\in C^1(\R)$ and $h$ is increasing. By \cite[Proposition 2.8]{Anzelloti}, we have that $\theta(\z,Dh(u))=\theta(\z,Du)$, $|Du|$-a.e. {Observe that in this case, the measure $|Dh(u)|$ is absolutely continuous with respect to $|Du|$ and vice versa. Then, a property holds $|Du|$-a.e. iff it holds $|Dh(u)|$ a.e.} Therefore  by integration by parts (see \cite{Anzelloti}), we obtain \begin{eqnarray}\label{claim_final}
 \nonumber \langle v, h(u)\rangle &=&-\int_\Omega h(u){\rm div }\z\,dx=\int_\Omega d(\z,Dh(u))\\ \nonumber &=&\int_\Omega \theta(\z,Dh(u))\,d|Dh(u)| = \int_\Omega \theta(\z,Du)\,d|Dh(u)|\\ &\stackrel{\eqref{zDu}}=& |Dh(u)|(\Omega).
   \end{eqnarray}
   In the general case, we approximate $h$ by a sequence of $C^1$ increasing functions $h_n(t):=h\star\rho_{\frac{1}{n}}{(t)}+\frac{t}{n}$, $0\leq \rho_{\frac{1}{n}}$ being a symmetric mollifier. Therefore, it is easy to check that $h_n(u)$ strictly converges to $h(u)$ and we finish the proof {by using \eqref{claim_final}}.
\end{proof}

%
%
%
%

\end{document}